\documentclass[a4paper,12pt]{article}
\usepackage{amssymb,amsmath,amsthm,latexsym}
\usepackage[croatian,english]{babel}

\usepackage{color}

\usepackage[a4paper,margin=1in]{geometry} 
\usepackage{verbatim}
\usepackage{dsfont}

\theoremstyle{plain}
\newtheorem{thm}{Theorem}[section]
\newtheorem{prop}[thm]{Proposition}
\newtheorem{corollary}[thm]{Corollary}
\newtheorem{lemma}[thm]{Lemma}

\theoremstyle{definition}

\newtheorem{remark}[thm]{Remark}

\numberwithin{equation}{section}

\newcommand{\I}{\mathds{1}}
\newcommand\R{\mathds{R}}
\newcommand\E{\mathds{E}}
\renewcommand\P{\mathds{P}}
\newcommand\N{\mathds{N}}
\newcommand\D{\mathds{D}}

\newcommand\FF{\mathcal{F}}

\newcommand\wh{\widehat}
\newcommand\wt{\widetilde}
\newcommand\ov{\overline}

\begin{document}

\title{\Large\bfseries A distributional equality for suprema of spectrally positive L\'evy processes}

\author{\itshape
    Ivana Ge\v{c}ek Tu{\dj}en
    \and
    \itshape Zoran Vondra\v{c}ek
}

\date{}

\maketitle

\begin{abstract}\noindent
Let $Y$ be a spectrally positive L\'evy process with $\E Y_1<0$, $C$ an independent subordinator with finite expectation, and $X=Y+C$. A curious distributional equality proved in \cite{HPSV04a} states that if $\E X_1<0$, then $\sup_{0\le t <\infty}Y_t$ and the supremum of $X$ just before the first time its new supremum is reached by a jump of $C$ have the same distribution. In this paper we give an alternative proof of an extension of this result and offer an explanation why it is true.



\end{abstract}

\section{Introduction}\label{sec-1}

Let $Y=(Y_t)_{t\ge 0}$ be a one-dimensional spectrally positive L\'evy process such that $Y_1$ is integrable and $\E Y_1<0$. By the law of large numbers, cf.~\cite[Theorem 36.5]{Sat}, $\lim_{t\to \infty}Y_t=-\infty$ a.s., and consequently $\ov{Y}_{\infty}:=\sup_{t\ge 0}Y_t <\infty$. Assume further that $C=(C_t)_{t\ge 0}$ is a subordinator without drift independent of $Y$ with jumps denoted by $\Delta C_t=C_t-C_{t-}$. By setting  $X_t:=Y_t+C_t$ we see that $X=(X_t)_{t\ge 0}$ is again a spectrally positive L\'evy process. Its supremum process is defined by $\ov{X}_t:=\sup_{0\le s \le t}X_s$. Let
\begin{equation}\label{e:def-sigma}
\sigma:=\inf\{t>0:\, \Delta C_t>\ov{X}_{t-}-X_{t-}\}
\end{equation}
be the first time the supremum of $X$ is reached by a jump of the subordinator $C$. As a consequence of spectral positivity of $Y$ it holds that $\sigma >0$ a.s.~(cf.~\cite[Theorem 4.1]{HPSV04a} and \cite[Theorem 2.1]{SV08}). Assume further that $C_1$ has finite expectation satisfying $\E X_1=\E Y_1+\E C_1 <0$. Then the following distributional equality was proved in \cite[Corollary 4.10]{HPSV04a}:
\begin{equation}\label{e:dist-eq}
\sup_{0\le t <\infty} Y_t \stackrel{d}{=} \sup_{0\le t<\sigma} X_t\, .
\end{equation}
Note that the right-hand side is the supremum of the process $X=Y+C$ just before the first time the new supremum of $X$ is reached by a jump of the subordinator $C$. As such, one might expect that its distribution depends on the subordinator $C$. A curious fact about \eqref{e:dist-eq} is that the right-hand side is independent of $C$ (as long as $\E Y_1+\E C_1<0$). The proof given in \cite{HPSV04a} does not reveal why this is so -- the equality \eqref{e:dist-eq} was obtained by deriving the Pollaczek-Khintchine formula for the overall supremum $\sup_{0\le t<\infty}X_t$ in two different ways and by equating factors in the Laplace transforms. The goal of this paper is to give an alternative proof of (a slight extension) of \eqref{e:dist-eq} which hopefully sheds more light on why this equality holds true and what are the limitations of further extensions of the formula. More precisely, we will prove the following result.

\begin{thm}\label{t:main}
Let $Y=(Y_t)_{t\ge 0}$ be a spectrally positive L\'evy process such that $Y_1$ is integrable and $\E Y_1<0$, let $C=(C_t)_{t\ge 0}$ be a subordinator without a drift independent of $Y$ such that $C_1$ is integrable and $\E Y_1 +\E C_1 \le 0$. If $X=Y+C$ and $\sigma$ is defined by \eqref{e:def-sigma}, then the distributional equality \eqref{e:dist-eq} holds true.
\end{thm}

Note that we extend \cite[Corollary 4.10]{HPSV04a} to the case when $\E X_1=0$. On the other hand, we also show that when $\E X_1 >0$ \eqref{e:dist-eq} is no longer valid.

The proof of Theorem \ref{t:main} is split into two parts. The first part deals with the case when both $Y$ and $C$ are compound Poisson processes. Then $Y$ can be written as $Y_t= -ct+Z_t $ where $Z=(Z_t)_{t\ge 0}$ is a subordinator with finite L\'evy measure such that $\E Z_1 <c$. Similarly, $X_t=-ct + Z_t +C_t$. Then $\sup_{0\le t<\infty}Y_t$ can be written as a sum of geometrically many overshoots leading to a new supremum. One minus the parameter of this geometric distribution is the probability that $Y$ goes above  level zero, while the distribution of each overshoot is equal to the distribution of the overshoot of level zero (conditional on the fact that level zero is reached). On the other hand, $\sup_{0\le t<\sigma}X_t$ is also a sum of geometrically many overshoots leading to a new supremum reached by a jump of $Z$. One minus the parameter of the geometric distribution is equal to the probability that $X$ goes over zero by a jump of $Z$, while each of the overshoots has the same distribution as the distribution of the overshoot of level zero (conditional on the fact that level zero is reached by a jump of $Z$). An explanation that the corresponding quantities in these two situations are equal relies on Tak\'acs' formula \cite[p.37]{Tak} and a fluctuation identity for spectrally one-sided L\'evy processes \cite[Corollary VII 3]{Ber}. In this part we use and extend the arguments from \cite[Section 2]{HPSV04b}. We further show that if $\E X_1>0$, then  $\sup_{0\le t<\infty}Y_t$ and $\sup_{0\le t<\sigma}X_t$ have different distributions. This part of the proof is explained in Section \ref{sec-2} of the paper.

In Section \ref{sec-3} we give two approximation results. The first one roughly says that if a L\'evy process $Y$ is a distributional limit of a sequence $(Y^{(n)})_{n\ge 1}$ of L\'evy processes, $C$ is a subordinator with finite L\'evy measure, and if \eqref{e:dist-eq} holds for approximating processes, then it also holds for the limiting process, see Proposition \ref{p:approx-1}. The second result is of a similar nature, only the subordinator $C$ with infinite L\'evy measure is approximated by a sequence of subordinators of finite L\'evy measures, see Proposition \ref{p:approx-2}. Both results rely on certain approximations of functions in Skorohod's space $\D$,  see Lemmas \ref{l:approx-CPP} and \ref{l:approx-general}. Proofs of the lemmas are quite technical and the reader may want to continue with the general case before delving into proofs.

In Section \ref{sec-4} we check that conditions required in approximation results are valid and give the proof of Theorem \ref{t:main} for the general case.  We finish the paper with a discussion on how essential is the assumption on spectral positivity of $Y$ for validity of Theorem \ref{t:main} and show that the theorem fails in case $Y$ does  not creep downwards (i.e.~it continuously crosses every level from above with probability zero, cf.~\cite[p.174]{Ber} for details).

\section{ The case of compound Poisson process}\label{sec-2}

Let $Z=(Z_t)_{t\ge 0}$ be a subordinator with no drift,  finite L\'evy measure $\nu_Z$ and the Laplace exponent $\phi_Z$ given by
$$
\phi_Z(\lambda)=\int_{(0,\infty)}(1-e^{-\lambda x})\, \nu_Z(dx)\, .
$$
Assume further that $\mu_Z:=\E Z_1=\int_0^{\infty}\nu_Z(x,\infty)\, dx <\infty$. Let $C=(C_t)_{t\ge 0}$ be another subordinator, independent of $Z$, with no drift, finite L\'evy measure $\nu_C$ and the Laplace exponent $\phi_C$. Assume also that $\mu_C:=\E C_1 < \infty$. Let $(\FF_t)_{t\ge 0}$ be the natural filtration generated by subordinators $Z$ and $C$, augmented in the usual way.
By independence, $Z$ and $C$ do not jump at the same time. We will tacitly use this fact throughout the paper.

Let $c>0$. The process $X=(X_t)_{t\ge 0}$ defined by $X_t:=-ct+Z_t+C_t$ is a spectrally positive L\'evy process such that $X_1$ is integrable and $\E X_1=-c+\mu_Z+\mu_C$. Note that $Z$ and $C$ have symmetric roles in $X$. The process $-X$ is a spectrally negative L\'evy process with the Laplace exponent $\psi=\psi_X$ defined by
$$
\E\left[e^{\lambda(-X_t)}\right]=e^{t \psi_X(\lambda)}\, ,\qquad \lambda >0\, ,
$$
and $\psi_X:[0,\infty)\to (-\infty,\infty)$ is strictly convex (cf.~\cite[p.188]{Ber}). Let $\Phi_X(0)$ be the largest solution of the equation $\psi_X(\lambda)=0$. Then $\psi_X:[ \Phi_X(0), \infty)\to [0,\infty)$ is a bijection, and its inverse is denoted by $\Phi_X$. Note that $\Phi_X(0)=0$ if and only if $\E X_1\le 0$. For $y\le 0$ let
$$
T_y^X:=\inf\{t\ge 0:\, X_t<y\}=\inf\{t\ge 0:\, X_t=y\}\, ,
$$
where the equality follows from the fact that $X$ is spectrally positive. We deduce from \cite[Theorem VII 1]{Ber} that $(T_y^X)_{y\le 0}$ is a (possibly killed) subordinator with the Laplace exponent $\Phi_X$. In particular,
\begin{equation}\label{e:hitting-y}
\P(T_y^X<\infty)=e^{\Phi_X(0)y}\, ,\quad y\le 0\, .
\end{equation}

Let $\tau_0^X:=\inf\{t>0:\, X_t>0\}$ be the first passage time of $X$ above the level zero. Note that at $\tau_0^X$ the process $X$ makes a jump over zero, and that either $Z$ or $C$ can make this jump. In the next proposition we compute the probability that the jump was made by $Z$ and the distribution of the overshoot. The same result was proved in \cite[Theorem 2.2 (a)]{HPSV04b} in case when $\E X_1 <0$.  A  related result is given in \cite[Theorem VII 17 (ii)]{Ber}.

\begin{prop}\label{p:crossing-0} Let $X_0=0$. For $y\le 0$ and $x>0$,
\begin{equation}\label{e:crossing-0-1}
\P\big (\tau_0^X<\infty, X_{\tau_0^X-}\in dy, X_{\tau_0^X}\in dx, \Delta X_{\tau_0^X}=\Delta Z_{\tau_0^X}\big)= \frac{1}{c} \, e^{\Phi_X(0)y} \nu_Z(-y+dx)\, dy\, .
\end{equation}
Consequently,
\begin{eqnarray}\label{e:crossing-0-2}
\lefteqn{\P\big (\tau_0^X<\infty,  X_{\tau_0^X}\in dx, \Delta X_{\tau_0^X}=\Delta Z_{\tau_0^X}\big)}\nonumber \\
&=&\frac{1}{c} \, \left( \nu_Z(x,\infty)-\Phi_X(0)e^{\Phi_X(0)x}\int_x^{\infty}e^{-\Phi_X(0)u}\nu_Z(u,\infty)\, du\right)\, dx
\end{eqnarray}
and
\begin{equation}\label{e:crossing-0-3}
\P\big (\tau_0^X<\infty, \Delta X_{\tau_0^X}=\Delta Z_{\tau_0^X}\big)= \frac{1}{c} \int_0^{\infty}e^{-\Phi_X(0)u} \nu_Z(u,\infty)\, du\, .
\end{equation}
\end{prop}
\begin{proof}
The proof relies on the following two results: The first one is (a version of) the remarkable formula due to Tak\'acs, see \cite[p.37]{Tak} which states that
\begin{equation}\label{e:takacs}
\P\left(\sup_{0\le s\le t} X_s >0 \, \Big | \, X_t\right)=1-\left(-\frac{X_t}{ct}\right)^+\, .
\end{equation}
The second result is the another remarkable identity  valid for one-sided L\'evy process, see  \cite[Corollary VII 3]{Ber}, which in our case says that
\begin{equation}\label{e:measures}
t \P(T_y^X\in dt)\, dy = (-y)\P(X_t \in dy)\, dt \, \quad \text{as measures on } [0,\infty)\times (-\infty, 0]\, .
\end{equation}
Recall that $\ov{X}_t=\sup_{0\le s \le t}X_s$. By use of the compensation formula applied to the two-dimensional Poisson point process $(\Delta Z_t, \Delta C_t)_{t\ge 0}$ with the characteristic measure concentrated on positive coordinate axes, we see that
\begin{eqnarray*}
\lefteqn{\P\big (\tau_0^X<\infty, X_{\tau_0^X-}\in dy, X_{\tau_0^X}\in dx, \Delta X_{\tau_0^X}=\Delta Z_{\tau_0^X}\big)}\\
&=&  \E\left(\sum_{0<t<\infty} \I_{(X_{t-}\in dy, \ov{X}_{t-}\le 0)} \I_{(\Delta Z_t\in -y+dx)}\right) \\
&=& \E\left(\int_0^{\infty} dt\, \I_{(X_{t-}\in dy, \ov{X}_{t-}\le 0)}\nu_Z(-y+dx)\right)\\
&=& \int_0^{\infty}dt\, \P(X_t\in dy, \ov{X}_t\le 0)\nu_Z(-y+dx)\, .
\end{eqnarray*}
The compensation formula was used in the second equality, while the last equality follows from the fact that $X_{t-}=X_t$ for a.e.~$t$.
By using first \eqref{e:takacs}, then \eqref{e:measures} and finally \eqref{e:hitting-y} we see that the last line is equal to
\begin{eqnarray*}
\lefteqn{\int_0^{\infty}dt\, \P(X_t\in dy) \frac{-y}{ct} \nu_Z(-y+dx)=\frac{1}{c} \int_0^{\infty}
\P(T_y^X\in dt)\, dy \, \nu_Z(-y+dx)} \\
&=&\frac{1}{c}\P(T_y^X<\infty)\nu_Z(-y+dx)\, dy =\frac{1}{c} \, e^{\Phi_X(0)y}  \nu_Z(-y+dx)\, dy\, .
\end{eqnarray*}
This proves \eqref{e:crossing-0-1}. By integrating over $y$ we obtain \eqref{e:crossing-0-2}. Indeed, define the measure $\rho$ on $(0,\infty)$ by
$$
\rho(A):=\frac{1}{c}\int_{-\infty}^0 e^{\Phi_X(0)y}\nu_Z(-y+A)\, dy=\frac{1}{c}\int_0^{\infty}e^{-\Phi_X(0)u}\nu_Z(u+A)\, du\, ,\quad A\subset (0,\infty)\, .
$$
When $A=(x,\infty)$, $x>0$, we get
\begin{eqnarray*}
\rho(x,\infty)&=&\frac1c \int_0^{\infty} e^{-\Phi_X(0)u} \nu(u+x,\infty)\, du\\
&=&\frac1c\,  e^{\Phi_X(0) x}\int_x^{\infty}e^{-\Phi_X(0)u}\nu_Z(u,\infty)\, du\, .
\end{eqnarray*}
It follows that the measure $\rho$ has a density $\rho(x)=-\frac{d}{dx}\rho(x,\infty)$. By differentiating we get that
$$
\rho(x)=\frac1c \nu_Z(x,\infty)-\Phi_X(0)\rho(x,\infty)
$$
which proves \eqref{e:crossing-0-2}. Finally,
$$
\P\big (\tau_0^X<\infty, \Delta X_{\tau_0^X}=\Delta Z_{\tau_0^X}\big)=\rho(0,\infty)=\frac{1}{c} \int_0^{\infty}e^{-\Phi_X(0)u} \nu_Z(u,\infty)\, du\, ,
$$
proving \eqref{e:crossing-0-3}.
\end{proof}

By applying Proposition \ref{p:crossing-0} to the process $Y_t:=-ct+Z_t$ the following identities follow:
\begin{align}
& \P\big (\tau_0^Y<\infty, Y_{\tau_0^Y-}\in dy, Y_{\tau_0^Y}\in dx \big)= \frac{1}{c} \, e^{\Phi_Y(0)y} \nu_Z(-y+dx)\, dy \label{e:crossing-0-1y}\\
& \P\big (\tau_0^Y<\infty,  Y_{\tau_0^Y}\in dx \big)
=\frac{1}{c} \, \left( \nu_Z(x,\infty)-\Phi_Y(0)e^{\Phi_Y(0)x}\int_x^{\infty}e^{-\Phi_Y(0)u}\nu_Z(u,\infty)\, du\right)\, dx \label{e:crossing-0-2y}\\
& \P\big (\tau_0^Y<\infty\big)= \frac{1}{c} \int_0^{\infty}e^{-\Phi_Y(0)y} \nu_Z(y,\infty)\, dy\, .\label{e:crossing-0-3y}
\end{align}

\begin{corollary}\label{c:crossing-0}
Suppose that $\mu_Z+\mu_C\le c$. Then
\begin{equation}\label{e:crossing-0-1e}
\P\big (\tau_0^X<\infty,  X_{\tau_0^X}\in dx, \Delta X_{\tau_0^X}=\Delta Z_{\tau_0^X}\big)=\P\big (\tau_0^Y<\infty,  Y_{\tau_0^Y}\in dx \big)=\frac{1}{c}\, \nu_Z(x,\infty)\, dx\, ,
\end{equation}
\begin{equation}\label{e:crossing-0-2e}
\P\big (\tau_0^X<\infty, \Delta X_{\tau_0^X}=\Delta Z_{\tau_0^X}\big)=\P\big (\tau_0^Y<\infty\big)=\frac{\mu_Z}{c}\, .
\end{equation}
and consequently
\begin{equation}\label{e:crossing-0-3e}
\P\big (  X_{\tau_0^X}\in dx \, \big| \tau_0^X<\infty,\Delta X_{\tau_0^X}=\Delta Z_{\tau_0^X}\big)=\P\big (Y_{\tau_0^Y}\in dx \, \big|\, \tau_0^Y<\infty\big)=\frac{\nu_Z(x,\infty)\, dx}{\mu_Z}\, .
\end{equation}
\end{corollary}
\begin{proof}
The assumption implies that $\E X_1 \le 0$, hence  $\Phi_X(0)=\Phi_Y(0)=0$. Now the first equality follows from \eqref{e:crossing-0-2} and \eqref{e:crossing-0-2y}, the second one is the consequence of \eqref{e:crossing-0-3}, \eqref{e:crossing-0-3y} and the fact that $\mu_Z=\int_0^{\infty}\nu_Z(x,\infty)\, dx$, while the third is immediate from the first two.
\end{proof}

We are now ready to prove Theorem \ref{t:main} in the compound Poisson case. As above, $Y_t=-ct+Z_t$, $X_t=-ct+Z_t+C_t$ where $\mu_Z+\mu_C\le c$. Since $\E Y_1<0$, $\sup_{t\ge 0}Y_t <\infty$ a.s. This supremum can be written as the geometric sum of overshoots leading to a new supremum. The distribution of the overshoots and the parameter of the geometric random variable (i.e.~the probability of reaching the new supremum) are given by \eqref{e:crossing-0-2y} and \eqref{e:crossing-0-3y}. On the other hand, $\sup_{0\le t <\sigma}X_t$ is also a geometric sum of those overshoots leading to the new supremum of $X$ obtained by jumps of $Z$. The distribution of such overshoots and the parameter of the geometric random variable are by \eqref{e:crossing-0-2e} and \eqref{e:crossing-0-3e} equal as before. This shows that $\sup_{t\ge 0}Y_t$ and $\sup_{0\le t <\sigma}X_t$ have the same distribution. We now make these arguments more precise.

\medskip
\noindent
\emph{Proof of Theorem \ref{t:main} -- compound Poisson process case.} Let $\tau^{(0)}=0$, $\tau^{(1)}=\inf\{t>0:\, Y_t>\ov Y_{t-}\}$ and for $n\ge 2$ define $\tau^{(n)}=\inf\{t>\tau^{(n-1)}:\, Y_t>\ov Y_{t-}\}$ on $\{\tau^{(n-1)}<\infty\}$. Note that $\tau^{(1)}=\tau_0^Y$. On $\{\tau^{(n)}<\infty\}$ define $I_n:= Y_{\tau^{(n)}}-\ov{Y}_{\tau^{(n-1)}}$, and let $N:=\max\{n:\, \tau^{(n)}<\infty\}$. Then
\begin{equation}\label{e:distribution-1}
\ov{Y}_{\infty}
=\sup_{t\ge 0}Y_t=\sum_{n=1}^N I_n
\end{equation}
(with the convention that $\sum_{n=1}^0 = 0$). By the strong Markov property $N$ has geometric distribution with parameter $\P(\tau^{(1)}=\infty)=1-\frac{\mu_Z}{c}$,  and for all $n\ge 1$, conditionally on $\tau^{(n)}<\infty$, $I_n$ has the distribution $\frac{1}{\mu_Z}\nu_Z(x,\infty)\, dx$, see \eqref{e:crossing-0-3e}.

On the other hand, let $\varsigma^{(0)}=0$, and let
$$
\wt{\varsigma} =\wt{\varsigma}^{(1)}:=\inf\{t>0:\, X_t>\ov{X}_{t-},\,  \Delta X_t=\Delta Z_t\}
=\inf\{t>0:\, \Delta Z_t>\ov{X}_{t-}-X_{t-}\}
$$
be the first time the new supremum of $X$ is reached by the jump of $Z$. Inductively, for $n\ge 2$  we define  $\wt{\varsigma}^{(n)}:=\inf\{t> \wt{\varsigma}^{(n-1)}:\,  X_t>\ov{X}_{t-},\,  \Delta X_t=\Delta Z_t\}$ on $\{ \wt{\varsigma}^{(n-1)}<\infty\}$. We are interested in times $\wt{\varsigma}^{(n)}$ only if they occur before $\sigma$. Hence we define $\varsigma^{(n)}:=\wt{\varsigma}^{(n)}\I_{(\wt{\varsigma}^{(n)}<\sigma)}+\infty\I_{(\wt{\varsigma}^{(n)}>\sigma)}$, $n\ge 1$.
Let further $J_n:= X_{\varsigma^{(n)}}-\ov{X}_{\varsigma^{(n-1)}}$ on $\{\varsigma^{(n)}<\infty\}$, and $M:=\max\{n:\, \varsigma^{(n)}<\infty\}$. Then
\begin{equation}\label{e:distribution-2}
\sup_{0\le t<\sigma}X_t=\sum_{n=1}^{M}J_n\, .
\end{equation}
Again, by the strong Markov property at stopping times $\varsigma^{(n)}$, $M$ has geometric distribution with parameter $\P(\varsigma^{(1)}=\infty)=1-\P(\tau_0^X<\infty,  \Delta X_{\tau_0^X}=\Delta Z_{\tau_0^X})=1-\frac{\mu_Z}{c}$, see \eqref{e:crossing-0-2e}. Further, by \eqref{e:crossing-0-3e}, for all $n\ge 1$, conditionally on $\varsigma^{(n)}<\infty$, $J_n$ has the distribution $\frac{1}{\mu_Z}\nu_Z(x,\infty)\, dx$. This finishes the proof. \qed

\begin{prop}\label{p:distributions-ne}
Assume that $\mu_Z<c<\mu_Z+\mu_C$. Then $\sup_{t\ge 0} Y_t$ and $\sup_{0\le t <\sigma} X_t$ have different distributions.
\end{prop}
For the proof we need the following simple result.
\begin{lemma}\label{l:lt-ne}
Let $(\xi_n)_{n\ge 1}$ be an i.i.d.~sequence of strictly positive random variables, $S_n=\xi_1+\cdots +\xi_n$, and let $N$ be an independent geometric random variable with parameter $1-\rho\in (0,1)$. Similarly, let $(\eta_n)_{n\ge 1}$ be an i.i.d.~sequence of strictly positive random variables, $T_n=\eta_1+\cdots +\eta_n$, and let $M$ be an independent geometric random variable with parameter $1-\varrho\in (0,1)$. If $S_N\stackrel{d}{=}T_M$ , then $\rho=\varrho$ and $\xi_1\stackrel{d}{=}\eta_1$.
\end{lemma}
\begin{proof}
Let $f(\lambda)=\E \left[e^{-\lambda \xi_1}\right]$ and $g(\lambda)=\E \left[e^{-\lambda \eta_1}\right]$. Then the Laplace transforms of $S_N$ and $T_M$ are given by
$$
\E\left[e^{-\lambda S_N}\right]=\frac{1-\rho}{1-\rho f(\lambda)}\, ,\qquad \E\left[e^{-\lambda T_M}\right]=\frac{1-\varrho}{1-\varrho g(\lambda)}\, .
$$
By the assumption, these two Laplace transforms are equal. By letting $\lambda \to \infty$ and using that $\lim_{\lambda\to \infty}f(\lambda)=\lim_{\lambda\to \infty}g(\lambda)=0$, we first get that $\varrho=\rho$, and then $g=f$.
\end{proof}

\noindent
\emph{Proof of Proposition \ref{p:distributions-ne}.} We use the notation from the proof of Theorem \ref{t:main} given above. The representations \eqref{e:distribution-1} and \eqref{e:distribution-2} are still valid. On the other hand, by the assumption we have that $\Phi_Y(0)=0$ while $\Phi_X(0)>0$. Therefore,
\begin{eqnarray*}
\P(\varsigma^{(1)}=\infty)&=&1-\P(\tau_0^X<\infty,  \Delta X_{\tau_0^X}=\Delta Z_{\tau_0^X})\\
&=&1-\frac{1}{c} \int_0^{\infty}e^{-\Phi_X(0)y} \nu_Z(y,\infty)\, dy \\
&> & 1-\frac{1}{c} \int_0^{\infty} \nu_Z(y,\infty)\, dy\\
&=&1-\frac{\mu_Z}{c}=\P(\tau^{(1)}=\infty)\, .
\end{eqnarray*}
The claim now follows from Lemma \ref{l:lt-ne}. \qed

\begin{remark}
(a) Assume that $\mu_C=\E C_1=\infty$ so that $\mu_Z<c<\mu_Z+\mu_C$. It is easy to see that Proposition \ref{p:crossing-0}  is still valid, hence  we conclude that Proposition \ref{p:distributions-ne} also holds.

\noindent
(b) Assume that $\mu_Z>c$. Then $Y$ drifts to $+\infty$, hence $\P(\tau_0^Y<\infty)=1$. We check that \eqref{e:crossing-0-3y} also gives this result. Indeed, since $\Phi_Y(0)>0$ solves the equation $\psi_Y(\lambda)=c\lambda -\phi_Z(\lambda)=0$, we have that $c\Phi_Y(0)= \phi_Z(\Phi_Y(0))$. By use of
$$
\phi_Z(\lambda)=\int_{(0,\infty)}(1-e^{-\lambda x})\, \nu_Z(dx)=\lambda \int_0^{\infty}e^{-\lambda x}\nu_Z(x,\infty)\, dx\, ,
$$
we see that
$$
c=\frac{1}{\Phi_Y(0)}\,  \phi_Z(\Phi_Y(0))=\int_0^{\infty}e^{-\Phi_Y(0)x}\nu_Z(x,\infty)\, dx\, .
$$
\end{remark}

\section{Two approximation results}\label{sec-3}
Let $\D=\D([0,\infty),\R)$ denote the space of all functions $x:[0,\infty)\to \R$ that are right-continuous and have left limits. Endowed with the Skorohod $J_1$-topology $\D$ becomes a Polish space (cf.~\cite[Chapter VI]{JS}). Let $\ov{x}(t):=\sup_{0\le s \le t}x(s)$.

\begin{lemma}\label{l:approx-CPP}
Assume that $(y_n)_{n\ge 1}\subset \D$, $y\in \D$ and $y_n\to y$ in $\D$. Let $c\in \D$ be of the form
\begin{equation}\label{e:form-of-c}
c(t)=\sum_{k\ge 1}\Delta c(s_k)\I_{(s_k\le t)}
\end{equation}
where $0<s_1<s_2<\dots $, $\lim_k{s_k}=\infty$, and $\Delta c(s_k)>0$ for all $k\ge 1$. Let $x:=y+c$, $x_n:=y_n+c$, $n\ge 1$, and assume that
\begin{equation}\label{e:no-sim-jumps}
\Delta y(s_k)=\Delta y_n(s_k)=0 \quad \text{for all }n\ge 1 \text{ and }k\ge 1\, ,
\end{equation}
and
\begin{equation}\label{e:strict-overshoot}
\Delta c(s_k)\neq \ov{x}(s_k-)-x(s_k-)\quad \text{for all }k\ge 1\, .
\end{equation}
Define
\begin{eqnarray}
\sigma&:=&\inf\{t>0:\, \Delta c(t)> \ov{x}(t-)-x(t-)\}\, ,\label{e:def-sigma-D}\\
\sigma_n &:=&\inf\{t>0:\, \Delta c(t)> \ov{x}_n(t-)-x_n(t-)\}\, \label{e:def-sigma-n-D}.
\end{eqnarray}
Then $\sigma=\lim_{n\to \infty}\sigma_n$ and $\ov{x}(\sigma-)=\lim_{n\to \infty}\ov{x}_n(\sigma_n-)$.
\end{lemma}
\begin{remark}\label{r:approx-CPP}
Note that the function $c$ is a typical realization of the subordinator $C$ with finite L\'evy measure. Assumption \eqref{e:no-sim-jumps} says that $y$ and $y_n$ do not jump at times when $c$ has a jump. Finally, assumption  \eqref{e:strict-overshoot} says that no jump of $c$ will hit the exact level of the current maximum of the function $x$.
\end{remark}

\noindent
\emph{Proof of Lemma \ref{l:approx-CPP}.}
We first note that $y_n(t)\to y(t)$ at every continuity point of $y$. Further, by \cite[Proposition VI 2.2, Proposition VI 2.1(a)]{JS} and the assumption \eqref{e:no-sim-jumps} we have $x_n\to x$ in $\D$.

Since $x_n=y_n$ and $x=y$ on $[0,s_1)$, and since $y$ and $y_n$ are continuous at $s_1$, it follows that
\begin{eqnarray*}
\ov{x}(s_1-)=\ov{y}(s_1-)\, &\qquad & x(s_1-)=y(s_1-)=y(s_1)\, ,\\
\ov{x}_n(s_1-)=\ov{y}_n(s_1-)\, &\qquad & x_n(s_1-)=y_n(s_1-)=y_n(s_1)\, .
\end{eqnarray*}
Further, $y(s_1)=\lim_{n\to \infty}y_n(s_1)$ implies that
\begin{equation}\label{e:cont-sigma-1}
x(s_1-)=\lim_{n\to \infty}x_n(s_1-)\, .
\end{equation}
By continuity of $y$ at $s_1$  it follows from \cite[Proposition VI 2.4]{JS} that $\ov{y}(s_1)=\lim_{n\to \infty}\ov{y}_n(s_1)$. Again by continuity of $y$ at $s_1$ we see that $\ov{y}(s_1)=\ov{y}(s_1-)$ and similarly for $\ov{y}_n$. We conclude that
\begin{equation}\label{e:cont-sigma-2}
\ov{x}(s_1-)=\lim_{n\to \infty}\ov{x}_n(s_1-)\, .
\end{equation}
Now \eqref{e:cont-sigma-1} and \eqref{e:cont-sigma-2} give together that
\begin{equation}\label{e:cont-sigma-3}
\ov{x}(s_1-)-x(s_1-)=\lim_{n\to \infty}\big(\ov{x}_n(s_1-)-x_n(s_1-)\big)\, .
\end{equation}

Suppose that $\sigma=s_1$, i.e.~$\Delta c(s_1)>\ov{x}(s_1-)-x(s_1-)$. It follows from \eqref{e:cont-sigma-3} that there exists $n'_1\in \N$ such that for all $n\ge n'_1$ it holds that $\Delta c(s_1)>\ov{x}_n(s_1-)-x_n(s_1-)$. Therefore, $\sigma_n=s_1$ for all $n\ge n'_1$, and $\ov{x}_n(\sigma_n-)=\ov{x}_n(s_1-)$. In particular, it holds that $\lim_{n\to \infty}\sigma_n=s_1=\sigma$, and $\lim_{n\to \infty}\ov{x}_n(\sigma_n-)=\lim_{n\to \infty}\ov{x}_n(s_1-)=\ov{x}(s_1-)=\ov{x}(\sigma-)$.

Suppose now that $\Delta c(s_1)<\ov{x}(s_1-)-x(s_1-)$ which is by \eqref{e:strict-overshoot} same as $\sigma\neq s_1$. Then there exists $n_1\in \N$ such that for all $n\ge n_1$ it holds that $\Delta c(s_1)<\ov{x}_n(s_1-)-x_n(s_1-)$, i.e.~$\sigma_n\neq s_1$.

So far we have shown that if  $\sigma=s_1$, the claim of the lemma is true. If $\sigma\neq s_1$, we consider the interval $[0,s_2]$. Set $y^{(1)}=y$, $y^{(1)}_n=y_n$ and $c^{(1)}=c$. Define $y^{(2)}$, $y^{(2)}_n$ and $c^{(2)}$ in the following way:
\begin{eqnarray*}
y^{(2)}(t)&=&y^{(1)}(t)+\Delta c(s_1)\, ,\\
y^{(2)}_n(t)&=&y^{(1)}_n(t)+\Delta c(s_1)\, ,\\
c^{(2)}(t)&=&c^{(1)}(t)-\Delta c(s_1)=\sum_{j\ge 2}\Delta c(s_j)\, ,
\end{eqnarray*}
(the jump $\Delta c(s_1)$ is moved from $c$ to $y$ and $y_n$).
Then $y^{(2)}, y^{(2)}_n\in \D$, $\lim_{n\to \infty}y^{(2)}_n = y^{(2)}$ in $\D$, $c^{(2)}$ has the same form as $c^{(1)}$ (piecewise constant with positive jumps in $s_2<s_3< \dots $), and it holds that
$$
x=y^{(2)}+c^{(2)} \quad \textrm{and }\quad x_n=y^{(2)}_n+c^{(2)}\, .
$$
The functions $y^{(2)}$ and $y^{(2)}_n$ satisfy the assumption \eqref{e:no-sim-jumps} for all $k\ge 2$: $\Delta y^{(2)}(s_k)=\Delta y^{(2)}_n(s_k)=0$, $n\ge 1$, $k\ge 2$. The assumption \eqref{e:strict-overshoot} is also valid if $c=c^{(1)}$ is replaced by $c^{(2)}$. In the same way as in the first part of the proof we conclude that
\begin{eqnarray}
& &  \ov{x}(s_2-)=\lim_{n\to \infty}\ov{x}_n(s_2-)\, ,\label{e:cont-sigma-4}\\
& & \ov{x}(s_2-)-x(s_2-)=\lim_{n\to \infty}\big(\ov{x}_n(s_2-)-x_n(s_2-)\big)\, .\label{e:cont-sigma-5}
\end{eqnarray}
Suppose that $\sigma=s_2$, i.e.~$\Delta c(s_2)=\Delta c^{(2)}(s_2)>\ov{x}(s_2-)-x(s_2-)$. It follows from \eqref{e:cont-sigma-5} that there exists $n'_2\in \N$ such that for all $n\ge n'_2$ it holds $\Delta c(s_2)>\ov{x}_n(s_2-)-x_n(s_2-)$. Since $\sigma \neq s_1$, we have that $\sigma_n\neq s_1$ for $n\ge n_1$, hence for $n\ge n_1\vee n'_2$ it holds that $\sigma_n=s_2$. This immediately implies  $\lim_{n\to \infty}\sigma_n=s_2=\sigma$. From \eqref{e:cont-sigma-4} we conclude that $\lim_{n\to \infty}\ov{x}_n(\sigma_n-)=\lim_{n\to \infty}\ov{x}_n(s_2-)=\ov{x}(s_2-)=\ov{x}(\sigma-)$.

If $\Delta c(s_2)=\Delta c^{(2)}(s_2)<\ov{x}(s_2-)-f(s_2-)$, then there exists $n_2\in \N$ such that for all $n\ge n_2$ it holds that $\Delta c(s_2)<\ov{x}_n(s_2-)-x_n(s_2-)$, i.e.~$\sigma_n\neq s_2$.

The proof continuous by induction. \qed

\medskip
If $Y$, $Y^{(n)}$, $n\ge 1$, are L\'evy processes, we will write $Y^{(n)}\Rightarrow Y$ for the weak convergence of induced probability measures on $\D$. We use the analogous notation for the weak convergence of random variables.
\begin{prop}\label{p:approx-1}
Assume that $Y$ is a L\'evy process with infinite L\'evy measure and $(Y^{(n)})_{n\ge 1}$ a sequence of L\'evy processes such that $Y^{(n)}\Rightarrow Y$. Let $C$ be an independent subordinator with finite L\'evy measure. Define $X:=Y+C$, $X^{(n)}:=Y^{(n)}+C$, $n\ge 1$, and let
$$
\sigma^{(n)}:=\inf\{t>0:\, \Delta C_t> \ov{X}^{(n)}_{t-}-X^{(n)}_{t-}\}\, .
$$
If
\begin{equation}\label{e:approx-1-a}
\sup_{0\le t<\infty}Y^{(n)}_t\Rightarrow \sup_{0\le t <\infty}Y_t
\end{equation}
and
\begin{equation}\label{e:approx-1-b}
\sup_{0\le t<\infty}Y^{(n)}_t \stackrel{d}{=}\sup_{0\le t <\sigma^{(n)}}X^{(n)}_t \qquad \text{for all }n\ge1 \, ,
\end{equation}
then also
$$
\sup_{0\le t<\infty}Y_t \stackrel{d}{=}\sup_{0\le t <\sigma}X_t\, .
$$
\end{prop}
\begin{proof}
Since $\D$ is separable, it follows from Skorohod's representation theorem, see \cite[Theorem 6.7]{Bil}), that we can assume that processes $Y$ and $Y^{(n)}$, $n\ge 1$, are all defined on the same probability space $(\Omega, \FF, \P)$ and that $Y^{(n)}(\omega)\to Y(\omega)$ in $\D$ for every $\omega \in \Omega$. We note that here $Y^{(n)}(\omega)$ and $Y(\omega)$ are regarded as functions in $\D$. Without loss of generality we may assume that $C$ is defined on the same probability space $(\Omega, \FF, \P)$ and is independent of $Y$ and $(Y^{(n)})_{n\ge 1}$. Clearly, for a.e.~$\omega \in \Omega$, the function $C(\omega)$ is of the form \eqref{e:form-of-c}.
Since $C$ is independent of $Y$ and $Y^{(n)}$, $n\ge 1$,
the assumption \eqref{e:no-sim-jumps} holds $\P$-almost surely.
Further, since $X$ has infinite L\'evy measure, the assumption \eqref{e:strict-overshoot} holds $\P$-a.s.~by \cite[Proposition VI 4]{Ber}. We deduce from Lemma \ref{l:approx-CPP} that
$$
\sup_{0\le t <\sigma^{(n)}}X^{(n)}_t \to \sup_{0\le t <\sigma}X_t\, \quad \text{a.s.}
$$
The claim now follows from assumptions \eqref{e:approx-1-a} and \eqref{e:approx-1-b}.
\end{proof}

\smallskip
\begin{lemma}\label{l:approx-general}
Let $y\in \D$ and let $c\in \D$ be a non-decreasing function without continuous part such that
\begin{equation}\label{e:no-sim-jumps-2}
\Delta y(s)\Delta c(s)=0\, \quad \text{for all }s>0\, .
\end{equation}
For $n\ge 1$ define
$$
c_n(t):=\sum_{0<s\le t} \Delta c(s) \I_{\left(\Delta c(s)>\frac1n\right)}\, .
$$
Let $x:=y+c$, $x_n:=y+c_n$, $n\ge 1$, and assume that
\begin{equation}\label{e:strict-overshoot-2}
\Delta c(t)\neq \ov{x}(t-)-x(t-)\quad \text{for all }t>0\, .
\end{equation}
Define
\begin{eqnarray}
\sigma&:=&\inf\{t>0:\,  \Delta c(t)> \ov{x}(t-)-x(t-)\}\, ,\label{e:def-sigma-D2}\\
\sigma_n &:=&\inf\{t>0:\, \Delta c_n(t)> \ov{x}_n(t-)-x_n(t-)\}\, \label{e:def-sigma-n-D2}.
\end{eqnarray}
Then $\sigma=\lim_{n\to \infty}\sigma_n$ and $\ov{x}(\sigma-)=\lim_{n\to \infty}\ov{x}_n(\sigma_n-)$ where $\ov{x}(0-):=0$.
\end{lemma}
\begin{proof}
First note that $c(t)=\sum_{0<s\le t}\Delta c(s)$. Further, $c_n\in \D$, the sequence $(c_n)_{n\ge 1}$ is non-decreasing and converges to $c$ uniformly on every finite interval $[0,t]$.

Case 1: $\sigma >0$.
Suppose that $\sigma <\infty$ and define $\eta:=\Delta c(\sigma)- (\ov{x}(\sigma-)-x(\sigma-))>0$. Let $0<\epsilon <\eta$ be arbitrary. In particular, $\Delta c(\sigma)>\eta >\epsilon$. Choose $n_0\in \N$ such that $1/n_0 <\epsilon$ and for all $n\ge n_0$ it holds
$$
c(\sigma)-c_n(\sigma)=\sum_{0<s\le \sigma}\Delta c(s) \I_{\left(\Delta c(s) \le \frac1n\right)}<\frac{\epsilon}{4}\, .
$$
Note that for  $n\ge n_0$, the function $c_n$ has a jump at $\sigma$: $\Delta c_n(\sigma)=\Delta c(\sigma)$. Further, on $[0,\sigma]$ we have
$$
x-\frac{\epsilon}{4}=y+c-\frac{\epsilon}{4}<y+c_n=x_n\le x\, .
$$
If $ s  \in [0,\sigma)$ is such that $x(s) > \ov{x}(\sigma-)-\epsilon/4$, then $x_n(s)>x(s )-\epsilon/4>\ov{x}(\sigma-)-\epsilon/2$, which implies
\begin{equation}\label{e:cont-sigma-6}
\ov{x}(\sigma-)-\frac{\epsilon}{2} \le \ov{x}_n(\sigma-) \le \ov{x}(\sigma-)\, .
\end{equation}
It also holds that
$$
x(\sigma-)-\frac{\epsilon}{4}\le x_n(\sigma-) \le x(\sigma-)\, .
$$
From the last two displays we see that for $n\ge n_0$
$$
\ov{x}(\sigma-)-x(\sigma-)-\frac{\epsilon}{2}<\ov{x}_n(\sigma-)-x_n(\sigma-)<\ov{x}(\sigma-)-x(\sigma-)+\frac{\epsilon}{2}\, .
$$
In particular, for all $n\ge n_0$
\begin{eqnarray*}
\Delta c_n(\sigma)-\big(\ov{x}_n(\sigma-)-x_n(\sigma-)\big)&=&\Delta c(\sigma)-\big(\ov{x}_n(\sigma-)-x_n(\sigma-)\big)\\
&>&\Delta c(\sigma)-\big(\ov{x}(\sigma-)-x(\sigma-)+\frac{\epsilon}{2}\big)\\
&>&\frac{\epsilon}{2}\, .
\end{eqnarray*}
This means that a new supremum of $x_n$ is reached by a jump of $c_n$ at time $\sigma$. Therefore, $\sigma_n\le \sigma$.

To prove the opposite inequality, suppose that $0<t<\sigma$. Together with the assumption \eqref{e:strict-overshoot-2} this implies that $\Delta c(t)- (\ov{x}(t-)-x(t-))<0$. Set $\epsilon=-\big(\Delta c(t)- (\ov{x}(t-)-x(t-))\big)$. The same argument as above implies that for all $n$ large enough it holds
$$
\Delta c_n(t)-\big(\ov{x}_n(t-)-x_n(t-)\big)<-\frac{\epsilon}{2}\, .
$$
This means that $\sigma_n>t$. We conclude that $\sigma_n \ge \sigma$. Note that this argument is valid also in the case when $\sigma=\infty$. Together with the first part of the proof this implies  $\sigma_n=\sigma$ for all sufficiently large $n$. In particular, $\sigma=\lim_{n\to \infty}\sigma_n$.

Suppose that $\sigma<\infty$. Since $\epsilon$ is arbitrary, it follows from \eqref{e:cont-sigma-6} that $\ov{x}(\sigma-)=\lim_{n\to \infty}\ov{x}_n(\sigma_n-)$. If $\sigma=\infty$, then $\ov{x}(\infty)=\lim_{n\to \infty}\ov{x}_n(\infty)$.

Case 2: $\sigma=0$. We claim that $\limsup_{n\to \infty}\sigma_n=0$. Suppose not, and let $\delta:=\limsup_{n\to \infty}\sigma_n>0$. Since $\sigma=0$, there exists $t\in (0,\delta/2)$ such that $\eta:=\Delta c(t)- (\ov{x}(t-)-x(t-))>0$. Let $0<\epsilon <\eta$ be arbitrary. Now we follow the first part of the proof of Case 1 replacing $\sigma$ by $t$ to conclude that there exists $n_0\in \N$ such that for all $n\ge n_0$
$$
\Delta c_n(t)-\big(\ov{x}_n(t-)-x_n(t-)\big)>\frac{\epsilon}{2}\, .
$$
It follows that $\sigma_n\le t<\delta/2$ for all $n\ge n_0$ which is a contradiction with $\limsup_{n\to \infty}\sigma_n=\delta$. Since $\ov{x}_n(\sigma_n-)\le \ov{x}(\sigma_n-)$, see \eqref{e:cont-sigma-6}, we conclude that
$$
\limsup_{n\to \infty}\ov{x}_n(\sigma_n-)\le \limsup_{n\to \infty}\ov{x}(\sigma_n-)=0\, .
$$
\end{proof}

Let $C=(C_t)_{t\ge 0}$ be a subordinator (without drift) with infinite L\'evy measure. For $n\ge 1$ define the process $C^{(n)}=(C^{(n)}_t)_{t\ge 0}$ by
\begin{equation}\label{e:approx-sub}
C^{(n)}_t:=\sum_{0<s\le t}\Delta C(s) \I_{\left(\Delta C(s)>\frac1n\right)}\, .
\end{equation}
Clearly, $C^{(n)}$ is a subordinator (without drift) with the finite L\'evy measure $\nu_n:=\nu_{|\left(\frac1n,\infty\right)}$.
\begin{prop}\label{p:approx-2}
Assume that $Y$ is a L\'evy process and $C$ an independent subordinator with infinite L\'evy measure. Define $X:=Y+C$, $X^{(n)}:=Y+C^{(n)}$, $n\ge 1$, and let
$$
\sigma^{(n)}:=\inf\{t>0:\, \Delta  C^{(n)}_t > \ov{X}^{(n)}_{t-}-X^{(n)}_{t-}\}\, .
$$
If
\begin{equation}\label{e:approx-2-b}
\sup_{0\le t<\infty}Y_t \stackrel{d}{=}\sup_{0\le t <\sigma^{(n)}}X^{(n)}_t \qquad \text{for all }n\ge1 \, ,
\end{equation}
then also
$$
\sup_{0\le t<\infty}Y_t \stackrel{d}{=}\sup_{0\le t <\sigma}X_t\,  ,
$$
where $\sup_{0\le t <\sigma}X_t:=0$ in case $\sigma=0$.
\end{prop}
\begin{proof}
In the same way as in the proof of Proposition \ref{p:approx-1} we see that the assumptions \eqref{e:no-sim-jumps-2} and \eqref{e:strict-overshoot-2} hold for a.e.~$\omega \in \Omega$. It follows from Lemma \ref{l:approx-general} that
$$
\sup_{0\le t <\sigma^{(n)}}X^{(n)}_t \to \sup_{0\le t <\sigma}X_t \quad \textrm{a.s.}
$$
Together with the assumption \eqref{e:approx-2-b} this proves the claim.
\end{proof}

\begin{remark}\label{r:approx-2}
{\rm
Note that if $Y$ and $C$ in Proposition \ref{p:approx-2} are such that $\sigma=0$ a.s.~and $\sup_{t\ge 0}Y_t$ is not identically zero, then \eqref{e:approx-2-b} cannot hold. We will come to this again at the end of Section \ref{sec-4}.
}
\end{remark}

\section{The general case}\label{sec-4}
Let $Y=(Y_t)_{t\ge 0}$ be a spectrally positive L\'evy process such that $\gamma:=\E Y_1<0$. The characteristic exponent of $Y$ will be denoted by $\Psi_Y$, that is $\E[\exp\{iz Y_t\}]=\exp\{t\Psi_Y(z)\}$. Then
$$
\Psi_Y(z)=-\frac12 az^2 +i\gamma z +\int_{(0,\infty)}\left(e^{izx}-1-izx\right)\, \nu(dx)\, ,
$$
where $a\ge 0$ is the diffusion coefficient. With the centering function $c(x)\equiv 1$, the L\'evy triplet of $Y$ is equal to $(a,\gamma,\nu)$, cf.~\cite[pp.38,39 and p.163]{Sat}. Throughout this section we assume that the L\'evy measure $\nu$ of $Y$ is infinite.

We start by recording the well-known fact that such a process can be approximated by a sequence of spectrally positive L\'evy processes $(Y^{(n)})_{n\ge 0}$ with finite L\'evy measures. This approximation is in the sense of weak convergence of one-dimensional distributions as well as weak convergence in the Skorohod space $\D$.

\begin{lemma}\label{l:approximation}
Let $Y=(Y_t)_{t\ge 0}$ be a spectrally positive L\'evy process such that $\gamma:=\E Y_1<0$. There exists a sequence $(Y^{(n)})_{n\ge 1}$ of spectrally positive L\'evy processes with finite L\'evy measures such that $\E Y_1^{(n)}=\gamma$, $Y_1^{(n)}\Rightarrow Y_1$ and $Y^{(n)}\Rightarrow Y$ in $\D$.
\end{lemma}
\begin{proof}
Let $(a,\gamma,\nu)$ be the L\'evy triplet of $Y$ (with the centering function $c(x)\equiv 1$). Let $(x_n)_{n\ge 1}$ be a sequence in $(0,1)$ such that $\lim_{n\to \infty} x_n=0$. For $n\ge 1$ we let $Y^{(n)}$ be the L\'evy process with the triplet $(0,\gamma, \nu_n)$ (again with the centering function $c(x)\equiv 1$) where
$$
\nu_n:=\nu_{|(\frac1n,\infty)}+\frac{a}{x_n^2}\delta_{x_n}\, .
$$
Clearly, $\nu_n$ is a finite measure concentrated on $(0,\infty)$ and $\E Y_1^{(n)}=\gamma$. It is straightforward to check that the characteristic exponent $\Psi_{Y^{(n)}}$ converges pointwise to $\Psi_Y$. This is clearly equivalent to the weak convergence $Y_1^{(n)}\Rightarrow Y_1$. Finally, the weak convergence of processes $Y^{(n)}\Rightarrow Y$ follows from \cite[Corollary~VII 3.6]{JS}.
\end{proof}
The approximating process $Y^{(n)}$ can be written in the form $Y^{(n)}_t=-c^{(n)}t+Z^{(n)}_t$ where $c^{(n)}>0$, $Z^{(n)}$ is a subordinator with finite L\'evy measure and no drift, and $-c^{(n)}+\E Z^{(n)}_1=\gamma<0$.

\begin{prop}\label{p:approx-sup}
Let $Y$ and $(Y^{(n)})_{n\ge 1}$ be as in Lemma \ref{l:approximation}. Then
$$
\sup_{0\le t <\infty} Y^{(n)}_t \Rightarrow \sup_{0\le t<\infty}Y_t\, .
$$
\end{prop}
\begin{proof}
Let $\psi$ denote the Laplace exponent of the spectrally negative dual process $\wh{Y}=-Y$ in the sense of \cite[p.188]{Ber}. Then $\psi(\lambda)=\Psi_Y(i\lambda)$, $\lambda \ge 0$. It is a straightforward consequence of \cite[VII (3), p.192]{Ber} that
$$
\E\left[\exp\left\{\lambda \inf_{0\le t<\infty}\wh{Y}_t\right\}\right]=-\gamma \frac{\lambda}{\psi(\lambda)}\, ,\quad \lambda >0\, ,
$$
(cf.~also \cite[(3.1)]{HPSV04a}). In terms of the process $Y$ this reads as
$$
\E\left[\exp\left\{-\lambda \sup_{0\le t<\infty}Y_t\right\}\right]=-\gamma \frac{\lambda}{\Psi_Y(i\lambda)}\, ,\quad \lambda >0\, .
$$
The same relation holds for the approximating processes:
$$
\E\left[\exp\left\{-\lambda \sup_{0\le t<\infty}Y^{(n)}_t\right\}\right]=-\gamma \frac{\lambda}{\Psi_{Y^{(n)}}(i\lambda)}\, ,\quad \lambda >0\, .
$$
Since $\Psi_{Y^{(n)}}\to \Psi_Y$ pointwise, we see that the Laplace transforms of $\sup_{0\le t<\infty}Y^{(n)}_t$ converge to the Laplace transform of $\sup_{0\le t<\infty}Y_t$. This proves the claim.
\end{proof}

\medskip
\noindent
\emph{Proof of Theorem \ref{t:main} -- general case.}
Let $Y$ be a spectrally positive L\'evy process with infinite L\'evy measure satisfying $\E Y_1< 0$.

We first consider  an independent subordinator $C$ (without drift) with finite L\'evy measure, and set $X:=Y+C$. By Lemma \ref{l:approximation} there exists a sequence  $(Y^{(n)})_{n\ge 1}$ of spectrally positive L\'evy processes with finite L\'evy measures such that $\E Y_1^{(n)}=\E Y_1$, $Y_1^{(n)}\Rightarrow Y_1$ and $Y^{(n)}\Rightarrow Y$ in $\D$. By Proposition \ref{p:approx-sup} we have $\sup_{0\le t <\infty} Y^{(n)}_t \Rightarrow \sup_{0\le t<\infty}Y_t$. Let $X^{(n)}:=Y^{(n)}+C$. Then \eqref{e:approx-1-b} is true by the proof of Theorem \ref{t:main} for the compound Poisson case given in Section \ref{sec-2}. Now it follows from Proposition \ref{p:approx-1} that $\sup_{0\le t<\infty}Y_t\stackrel{d}{=} \sup_{0\le t<\sigma}X_t$.

In the second step we take an independent subordinator $C$ (without drift) with infinite L\'evy measure and set $X:=Y+C$. For each $n\ge 1$ define the subordinator $C^{(n)}$ by \eqref{e:approx-sub} and let $X^{(n)}:=Y+C^{(n)}$. Then by what has been just proved we have that $\sup_{0\le t<\infty}Y_t\stackrel{d}{=} \sup_{0\le t<\sigma}X^{(n)}_t$. We finish the proof by invoking Proposition \ref{p:approx-2}. \qed

\bigskip

At the end of the paper we discuss briefly how essential is the assumption on spectral positivity of $Y$ for validity of Theorem \ref{t:main}.
Suppose that $Y$ is a general L\'evy process, not necessarily spectrally positive, which is not a sum of negative subordinator and negative drift, and such that $\E Y_1<0$. Then $\sup_{t\ge 0} Y_t$ is finite a.s.~and not identically zero. Let $C$ be an independent subordinator without drift, define  $X=Y+C$, and assume that $\sigma=\inf\{t> 0:\, \Delta C_t>\ov{X}_{t-}-X_{t-}\}=0$. Then clearly $\sup_{0\le t <\sigma}X_t=0$ and hence cannot be equal in distribution to $\sup_{t\ge 0} Y_t$. Moreover, if $C^{(n)}$ is a sequence of subordinators defined as in \eqref{e:approx-sub}, $X^{(n)}=Y+C^{(n)}$ and $\sigma^{(n)}=\inf\{t>0:\, \Delta C^{(n)}_t>\ov{X}^{(n)}_{t-}-X^{(n)}_{t-}\}$, then $\sigma^{(n)}>0$ and according to Remark \ref{r:approx-2}, it cannot hold that
$$
\sup_{0\le t<\infty}Y_t \stackrel{d}{=}\sup_{0\le t <\sigma^{(n)}}X^{(n)}_t \qquad \text{for all }n\ge1 \, .
$$
This means that there exists a subordinator $\wt{C}$ with finite L\'evy measure, independent of $Y$ such that if $\wt{X}=Y+\wt{C}$ and $\wt{\sigma}=\inf\{t> 0:\, \Delta \wt{C}_t>\ov{\wt{X}}_{t-}-\wt{X}_{t-}>0\}$, then $\sup_{t\ge 0}Y_t$ and $\sup_{0\le t<\wt{\sigma}}\wt{X}_t$ have different distributions. Hence Theorem \ref{t:main} cannot hold for $Y$ even in case of subordinators with finite L\'evy measure.

The necessary and sufficient condition for $\sigma=0$ was given in \cite[Theorem 2.1]{SV08}. Let $Y$ be a L\'evy process of unbounded variation, let $C$ be an independent subordinator with the L\'evy measure $\nu$, and let $X=Y+C$. Denote by $V(x)$ the renewal function of the descending ladder height process of $Y$ (i.e.~the ascending ladder height of the dual process $-Y$), cf.~\cite[III 1 and VI 1 ]{Ber}. Then $\sigma>0$ a.s.~if and only if
\begin{equation}\label{e:cond-sigma-f}
\int_0^1 V(x)\, \nu(dx)<\infty\, .
\end{equation}
In case $Y$ is spectrally positive, it holds that $V(x)=x$, hence \eqref{e:cond-sigma-f} is automatically satisfied because of integrability property of $\nu$. More generally, $Y$ creeps downwards (i.e.~with positive probability crosses every level from above continuously), if and only if there exists a constant $c>0$ such that $V(x)\le cx$ for all $x\ge 0$, see \cite[Theorem VI 19, pp.~174-175]{Ber}.

Assume that $Y$ is of unbounded variation such that it is not true that $V(x)\le cx$ for $x\in [0,1]$ for any constant $c>0$. Then it follows that $\lim_{x\to 0+}V(x)/x=+\infty$. Let
$$
\beta(x):=\inf_{0<t\le x}\frac{V(t)}{t}\, .
$$
Then $\beta$ is non-increasing and $\lim_{x\to 0+}\beta(x)=+\infty$. Denote by $\beta(dx)$ the measure on $(0,1)$ corresponding to the function $\beta$ and define
$$
\wt{\nu}(dx):=\frac{\beta(dx)}{\beta(x)^2}\, .
$$
By a change of variable it is easy to see that
$$
\int_0^1 \beta(x)\wt{\nu}(dx)=\int_0^1 \frac{\beta(dx)}{\beta(x)}=\infty \qquad \text{and} \qquad \int_0^1 \wt{\nu}(dx)=\int_0^1 \frac{\beta(dx)}{\beta(x)^2}<\infty \, .
$$
Finally, let $\nu(dx):=\wt{\nu}(dx)/x$. Then
$$
\int_0^1 V(x)\, \nu(dx)\ge \int_0^1 \beta(x)\, \wt{\nu}(dx)=\infty \qquad \text{and} \qquad \int_0^1 x\, \nu(dx) =\int_0^1\wt{\nu}(dx)<\infty\, .
$$
Hence $\nu$ is a L\'evy measure such that \eqref{e:cond-sigma-f} is not satisfied. If $C$ is a subordinator independent of $Y$ with  L\'evy measure $\nu$, then $\sigma=0$ a.s. We conclude that Theorem \ref{t:main} cannot hold for the L\'evy process $Y$.

To summarize, if $Y$ is of unbounded variation, $\E Y_1<0$, and $Y$ does not creep downwards, then there exists a subordinator $C$ with finite L\'evy measure such that $\sup_{t\ge 0}Y_t$ and $\sup_{0\le t<\sigma}X_t$ have different distributions.

\medskip
\noindent
{\bf Acknowledgement:} We thank the referees for careful reading of the paper and helpful remarks. This work has been supported in
part by Croatian Science Foundation under the project 3526.

\bigskip

{\bf Ivana Ge\v{c}ek Tu{\dj}en}

Department of Mathematics, University of Zagreb, Zagreb, Croatia

Email: \texttt{igecek@math.hr}

\bigskip

{\bf Zoran Vondra{\v{c}}ek}

Department of Mathematics, University of Zagreb, Zagreb, Croatia

Email: \texttt{vondra@math.hr}
\end{document}